\documentclass[fleqn,12pt]{article}

\usepackage[T1]{fontenc}
\usepackage{float}
\usepackage{units}
\usepackage{mathrsfs}
\usepackage{bm}
\usepackage{amsmath}
\usepackage{amssymb}
\usepackage{graphicx}
\usepackage{esint}
\usepackage[numbers]{natbib}

\usepackage{simplemargins}
\setallmargins{1in}

\newtheorem{theorem}{Theorem}[section]

\newtheorem{proposition}[theorem]{Proposition}
\newtheorem{corollary}[theorem]{Corollary}
\newcommand{\E}{\mathrm{E}}
\newcommand{\Var}{\mathrm{Var}}
\newcommand{\Cov}{\mathrm{Cov}}
\newenvironment{proof}[1][Proof]{\begin{trivlist}
\item[\hskip \labelsep {\bfseries #1}]}{\end{trivlist}}

\newcommand{\qed}{\nobreak \ifvmode \relax \else
      \ifdim\lastskip<1.5em \hskip-\lastskip
      \hskip1.5em plus0em minus0.5em \fi \nobreak
      \vrule height0.75em width0.5em depth0.25em\fi}

\begin{document}

\begin{center}
\begin{Large}
Distribution of a Non-parametric Wavelet-based Statistic for Functional Data
\end{Large}

\vspace{.2in}

S. B. Girimurugan\footnote{Corresponding author: sgirimurugan@fgcu.edu} \\
Department of Mathematics \\
Florida Gulf Coast University \\

\vspace{.2in}

Eric Chicken \\
Department of Statistics \\
Florida State University \\
\end{center}

\flushbottom 

\begin{small}
\abstract{Mathematical formulations and proofs for a wavelet based statistic employed in functional data analysis is elaborately discussed in this report. The propositions and derivations discussed here apply to a wavelet based statistic with hard thresholding. The proposed analytic distribution is made feasible only due to the assumption of normality. Since the statistic is developed for applications in high dimensional data analysis, the assumption holds true in most practical situations. In the future, the work here could be extended to address data that are non-Gaussian.
Aside from establishing a rigorous mathematical foundation for the distribution of the statistic, the report also explores a few approximations for the proposed statistic.
\\
\\
\begin{footnotesize}
\end{footnotesize}
}
\end{small}

\thispagestyle{empty}

\section{Introduction}

Functional data are obtained from experiments that result in a smooth curve as response. In its original context, such data are assumed to have repeated measurements of smooth curves sampled from a population of curves defined by a single parameter or finitely many parameters. Several techniques unique to handling such data have been developed and the interested reader can find those in \citep{Ramsay2005}. The data considered here are noisy functional responses where an underlying smooth curve is perturbed by additive noise. In addition, the assumed functional model is non-parametric because the underlying structure is not explained by coefficients of a parametric model. The only assumption made in regards to the underlying structure is that the functional response is $\ell_2$ integrable. Sample realizations of functional responses from `T' treatments can be given as,
\begin{eqnarray}
Y_{ijk}=f_{ij}(x_k)+\epsilon
\label{e1}
\end{eqnarray}
where $i=1,2,\ldots,T$, $j=1,2,\ldots,r_i$ and $k=1,2,\ldots,n$. Also, $r_i$ is the number of replicates of the functional response obtained from the $i^{th}$ treatment. Thus, $Y_{ijk} \in \mathbb{R}^n$ and $\epsilon$ is the n-dimensional additive Gaussian noise. It can be seen that the structure of equation \eqref{e1} poses a striking resemblance to models frequently employed in multi-variate statistics.
\\
Functional data, in fact, has many commonalities with multivariate statistical techniques. For instance, a d-dimensional response vector is a multivariate random variable  and  differences among such vectors are detected by multivariate testing models. A functional response in reality is also a multidimensional vector and one might assume that a multivariate test would be sufficient. However, since functional responses are discretized (as shown in \eqref{e1}), statistical independence of the individual random variables within the vector is not guaranteed. Moreover, as the dimension of the functional response approaches infinity multivariate tests suffer from the curse of dimensionality where the dimension of the response exceeds the number of replicates. Many authors have proposed testing methods which address these issues. Fan and Li \cite{Fan1996, Fan1998} proposed methods using Neyman's truncation method and Fourier Transforms. The method, known as HANOVA, offered ways to circumvent the problems encountered with increasing dimensions by appropriately choosing Fourier coefficients instead of the raw data itself. Wavelet transforms \cite{Daubechies1992} offer better time-frequency localization when compared to Fourier Transforms and methods in \cite{Fan1996,Fan1998} can be improved by using wavelet coefficients instead of Fourier coefficients. One approach using wavelet coefficients is discussed in \cite{rosner2000wavelet} which employs an ANOVA setup for functional data in the wavelet domain. The proposed statistic discussed here takes on a different approach when compared to the setup in \cite{rosner2000wavelet}. Aside from the definition of the statistic and its application in testing functional data, the primary focus of this report is in establishing a theoretical foundation for the statistic and its approximations.
\section{Terminology}
Suppose the actual functional response for the $i^{th}$ treatment is defined as,
\begin{eqnarray}
f_i(\textbf{s})=f(\textbf{s})+\mu(\textbf{s})
\label{e2}
\end{eqnarray}
where \textbf{s} is a unit of time (or space). $f(\textbf{s})$ is the mean response common to all treatments and $\mu(\textbf{s})$ is the effect introduced by the $i^{th}$ treatment. If the time domain ($\textbf{s}$) is sampled at $(n-1)$ equally spaced intervals, then $f_i(k), k=1,2,\ldots n$ represents the discretized version of $f_i(\textbf{s})$. Thus, \eqref{e1} refers to the sample realization of $f_{i}(k)$ with additive noise of its $j^{th}$ replicate. \\
\\
For a sample of functional responses, $Y_{i.}$ corresponds to the mean functional response of the $i^{th}$ treatment obtained from its $r_i$ replicates. $Y_{..}$  represents the overall mean of all averaged responses from the `i' treatments. An ANOVA type test statistic for the following test,
\begin{eqnarray}
H_0 &:& f_1=f_2=\ldots = f_T \\
H_1 &:& \exists k \ni f_k \neq f_j, j \in \lbrace 1,\ldots ,T\rbrace \setminus k
\label{e3}
\end{eqnarray}
is given as,
\begin{eqnarray}
\vartheta &=& \sum_{i=1}^{T} (Y_{i.}-\bar{Y_{..}})\Sigma^{-1}(Y_{i.}-\bar{Y_{..}})^\prime
\label{e4}
\end{eqnarray}
It can be understood that \eqref{e4} is a common multivariate test statistic and for functional data such a statistic has a non-central F distribution \cite{Ramsay2005, rosner2000wavelet}.
\section{Motivation}
The statistic in \eqref{e4} is dependent on the estimation of $\Sigma$. Although estimating $\Sigma$ is easier when independence is guaranteed in the functional response, correlated responses tend to make the estimation of $\Sigma$ complicated. From a computational standpoint, $\Sigma$  may not be sparse and inverting such a matrix might be computationally extensive in higher dimensions. The proposed statistic is computationally efficient in this regard. This is mostly due to the pyramid algorithm which is used to perform a discrete wavelet transform on a signal and the lack of a need to invert $\Sigma$. \\
Moreover, wavelet shrinkage schemes offer sparsity along with noise reduction. The noise removal properties are highly desired in wavelet estimation of functions and the inherent sparsity offered by them make wavelet transforms more appealing in high-dimensional data analysis. Also, since wavelet transforms offer decorrelation, the wavelet coefficients are uncorrelated even if the original response is correlated. The coherence of these properties naturally improve the power of a statistical test with better computational effectiveness.
\section{Discrete Wavelet Transform}
A discrete wavelet transform (DWT) transforms a function $f(x)$ such that,

\begin{eqnarray}
f(x)=\sum_{k=1}^{M_{0}}V_{j_0k}\phi_{0k}(x)+\sum_{j=j_0}^{\infty}\sum_{k=1}^{M_{j}}W_{jk}\psi_{jk}(x)
\label{e36}
\end{eqnarray}
That is, the function is represented as a sum of orthogonal wavelet functions. It can be noticed that the decomposition has a construct similar to a Fourier decomposition. The difference in equation \eqref{e36} is due to the wavelet basis used in place of the trigonometric basis. Despite a minor, subtle change in the definition, the DWT offers better time-frequency localization as a primary consequence of the wavelet basis. In \eqref{e36}, $j,k \in \mathbb{Z}$ where $`j$' refers to the level of decomposition and $`k$' refers to a location in time or space. $`j_0$' corresponds to the lowest level of decomposition.\\
 A full DWT of a functional response, $f(x) \in \mathbb{R}$, with dimension $n=2^J$ will have $J$ levels of decomposition with $j_0=0$. The topmost level (j=J) represents the functional response itself and $j=0,1,\ldots,J-1$. A partial DWT upto level $j_0$ will have $j=j_0,j_0+1\ldots,J-1$.

The coefficients are obtained using a familiar approach, akin to Fourier coefficients. Here,
\begin{eqnarray}
W_{jk} &=& \int_{-\infty}^{\infty} f(x)\psi_{jk} dx\\
V_{jk} &=& \int_{-\infty}^{\infty} f(x)\phi_{jk} dx
\label{e37}
\end{eqnarray}
where,
\begin{eqnarray}
\psi_{jk} &=& 2^{\nicefrac{j}{2}}\psi \left(2^{j}x-k\right)\\
\phi_{jk} &=& 2^{\nicefrac{j}{2}}\phi \left(2^{j}x-k\right)
\label{e37}
\end{eqnarray}
Since the dilations are carried out in powers of 2, DWT has a minimal requirement where the sample size of the response is a power of 2. Zero padding can be employed to meet this requirement in responses that have a sample size which is not a power of 2. Alternatively, reflection of the data can also be employed \cite{hollander2013nonparametric}. The transform can be accomplished efficiently and faster than FFT via the pyramid algorithm, and for a response vector $Y \in \mathbb{R}^n , n=2^J$, the coefficients at level `j' can be obtained as,
\begin{eqnarray}
W_j &=& \Psi_j Y \\
V_j &=& \Phi_j Y
\label{e11}
\end{eqnarray}
where $\Psi_j$ and $\Phi_j$ are the filter coefficient matrices at the $j^{th}$ level.  In equation \eqref{e11}, the high-pass filter matrix $\Psi_j, \Phi_j \in \mathbb{R}^{\nicefrac{N}{2^{(J-j)}} \times \nicefrac{N}{2^{(J-(j-1))}}}$ resulting in $W_j, V_j \in \mathbb{R}^{\nicefrac{N}{2^{(J-j)}}}$ . The pyramid algorithm \cite{Mallat1989} uses these high-pass ($W_j$) and low-pass ($V_j$) coefficients in a recursive, top-down approach across the levels. Thus, the final set of coefficients will be comprised mostly $W_j$ coefficients. That is,
\begin{eqnarray}
\Theta = [W_{(J-1)}, W_{(J-2)},\ldots, W_{j_0},V_{j_0}]^\prime
\label{e12}
\end{eqnarray}
for a partial DWT up to level $j_0$. Whereas, for a full DWT, the coefficient vector is given by,
\begin{eqnarray}
\Theta = [W_{(J-1)}, W_{(J-2)},\ldots, W_{0},V_{0}]^\prime
\label{e13}
\end{eqnarray}
The reader should be aware that the dimension of $\Theta$ in equations \eqref{e12} and \eqref{e13} is `n'. The dimension is reduced by selecting a subset of these coefficients while discarding others. The methods used to select this subset are collectively known as "wavelet shrinkage" or "wavelet thresholding" schemes \cite{Abramovich1996a,Donoho1995,Donoho1994a,cai1999adaptive}. Although several methods have been discussed in the past, the proposed statistic employs the hard, VisuShrink procedure \cite{Donoho1994} with the thresholding parameter given by,
\begin{eqnarray}
\lambda = \sigma\sqrt{2\ln(n)}
\label{e14}
\end{eqnarray}
where $\sigma$ is the noise parameter which can be estimated using the standard deviation or the median absolute deviation (MAD) which is a robust estimator \cite{Ogden1997} in wavelet applications. Let $\hat{\Theta}$ denote the vector of wavelet coefficients after a threshold.
\section{Wavelet Statistic for Functional Data}
The proposed wavelet based statistic to test differences in functional data for the test in \eqref{e3} is $\kappa_\eta$. Using the terminology discussed in the previous sections, the wavelet coefficients of $\hat{\sigma}^{-1}(\bar{Y}_{i.}-\bar{Y}_{..})$ for the $i^{th}$ treatment is given by,
\begin{eqnarray}
\Theta_i &=& [\theta_{i1},\theta_{i2},\ldots,\theta_{in}]^\prime
\label{e15}
\end{eqnarray}
The dimension of the coefficient vector in \eqref{e15} is reduced via wavelet thresholding with the $\lambda$ in \eqref{e14} . The resulting coefficients can be denoted as,
\begin{eqnarray}
\hat{\Theta}_i &=& [\theta_{i1},\theta_{i2},\ldots,\theta_{ik}]^\prime
\label{e151}
\end{eqnarray}
where, $k \ll n$.\\
Since wavelet transforms are energy preserving, the responses in $\vartheta$ from \eqref{e4}  should satisfy the following inequality as a consequence of Parseval's identity \cite{spiegel1965laplace}.
\begin{eqnarray}
\parallel \Theta_i \parallel_2^{2} &\leq & \parallel (\bar{Y}_{i.}-\bar{Y}_{..}) \parallel_2^{2}
\label{e16}
\end{eqnarray}
An obvious result from \eqref{e16} can be deduced for thresholded coefficients as,
\begin{eqnarray}
\parallel \hat{\Theta}_i \parallel_{2}^{2} &<& \parallel \Theta_i \parallel_2^{2}
\label{e17}
\end{eqnarray}
Using equations \eqref{e16} and \eqref{e17}, a wavelet based statistic can be obtained as,
\begin{eqnarray}
\kappa_\eta &=& \sum_{i=1}^{T} \hat{\Theta}_i^2 \\
&=& \sum_{i=1}^{T} \sum_{k=1}^{m_i} \hat{\theta}_{ik}^2, \quad \quad m_i \ll n
\label{e18}
\end{eqnarray}
where,
\begin{eqnarray}
\Theta_i &=& \mathbb{W} \left[ \hat{\gamma}^{-1}(\overline{Y}_{i.}- \overline{Y}_{..})\right]^\prime \mathbb{W} \left[ \hat{\gamma}^{-1}(\overline{Y}_{i.}- \overline{Y}_{..}) \right]
\label{e19}
\end{eqnarray}
 The normalizing factor in \eqref{e19}, $\hat{\gamma}$, is an estimator of functional variance. For the $i^{th}$ treatment, we have
\begin{eqnarray}
\E[\bar{Y}_{i.}-\bar{Y}_{..}]=0
\label{at1}
\end{eqnarray}
then,
\begin{eqnarray}
\Var[\bar{Y}_{i.}-\bar{Y}_{..}] & = & \E[(\bar{Y}_{i.}-\bar{Y}_{..})^2] \\
& = & \Var [\bar{Y}_{i.}]+ \Var[\bar{Y}_{..}]-2\Cov(\bar{Y}_{i.},\bar{Y}_{..})\\
& = & \sigma_{i}^2 + \dfrac{1}{t^2}\sum_{i=1}^{t}\sigma_{i}^2 - 2\Cov(\bar{Y}_{i.},\dfrac{1}{t}\sum_{j=1}^{t}\bar{Y}_{j.}) \\
& = & \sigma_{i}^2 + \dfrac{1}{t^2}\sum_{i=1}^{t}\sigma_{i}^2 -\dfrac{2}{t} \left( \sigma_{i}^2 + \sum_{j=1}^{t} \Cov(\bar{Y}_{i.},\bar{Y}_{j.}) \right), i \neq j \\
& = & \sigma_{i}^{2} (1 -\dfrac{2}{t})+ \dfrac{1}{t^2}\sum_{i=1}^{t}\sigma_{i}^2 -\dfrac{2}{t} \left(\sum_{j=1}^{t} \Cov(\bar{Y}_{i.},\bar{Y}_{j.}) \right), i \neq j 
\label{at2}
\end{eqnarray}
In general, if the total variance is $\sigma^2$, the  variance of the $i^{th}$ treatment with $r_i$ replicates is $ \dfrac{1}{r_i} \sigma^2$, thus,
\begin{eqnarray}
\hat{\gamma} &=& \sigma^2 \left[ \dfrac{1}{r_i} \left( \dfrac{(t-2)}{t} \right) + \dfrac{1}{t^2} \sum_{i=1}^{t} \dfrac{1}{r_i} \right] -\dfrac{2}{t} \rho_{ij}
\label{at3}
\end{eqnarray}
where, for $i \neq j$,
\begin{eqnarray}
\rho_{ij}= \Cov(\bar{Y}_{i.},\bar{Y}_{j.})
\label{at4}
\end{eqnarray}
and,
\begin{eqnarray}
\hat{\sigma}^2= \Var({Y}_{ijk}-\bar{Y}_{i.k})
\label{at5}
\end{eqnarray}
\section{Distribution of $\kappa_\eta^{(p_i,q_i,\lambda)}$}
In this section, we derive the distribution of the proposed statistic for the $i^{th}$ treatment, $\kappa_\eta^{(p_i,q_i,\lambda)}$.
\begin{corollary}
\label{e413}
If the responses are normally distributed, the $k^{th}$ wavelet coefficient from the $i^{th}$ treatment, $\theta_{ik}$, in the proposed statistic should satisfy,
\begin{eqnarray}
\theta_{ik} \sim N(0,1)
\label{e414}
\end{eqnarray}
Therefore, for the $i^{th}$ treatment, $\kappa_\eta^{(p_i,q_i,\lambda)}$ can be given as,
\begin{eqnarray}
\kappa_\eta^{(p_i,q_i,\lambda)} &=& \sum_{k=1}^{p_i} \hat{\theta}_{ik}^2 + \sum_{k=1}^{q_i} \theta_{ik}^2\\
\label{e415}
\end{eqnarray}
which is the sum of squared wavelet coefficients after a threshold of $\lambda$ and the sum of squared wavelet coefficients without a threshold. Also,
\begin{eqnarray}
\sum_{k=1}^{p_i} \hat{\theta}_{ik}^2 & \sim & \left] \chi_{p_i}^{2} \right[ _\lambda \\
\sum_{k=1}^{q_i} \theta_{ik}^2 & \sim & \chi_{q_i}^2
\label{e416}
\end{eqnarray}
where, $\left] \chi_{p_i}^{2} \right[ _\lambda $ is a truncated Chi-squared distribution with $p_i$ degrees of freedom.
\end{corollary}
The truncated chi-squared distribution in  Corollary \eqref{e413} is different from the conventional truncated distributions. The commonly used truncated normal random variable (X) has support  $X \in (\lambda_1,\lambda_2), -\infty < \lambda_1 < \lambda_2 < \infty$ and the random variable $Y=X^2$ will have a truncated chi-squared distribution with support $ (0,\lambda^2]$ where $\lambda = \max(\lambda_1,\lambda_2)$.\\
\begin{proposition}
The exact distribution of the statistic, $\kappa_\eta^{(p_i,q_i,\lambda)}$, defined in Corollary \eqref{e413} is,
\label{p01}
\begin{eqnarray}
f_{\kappa_\eta^{(p_i,q_i,\lambda)}}(x) &=&  f_C(x)\dfrac{1-\Phi_B(p_i\lambda^2,\nicefrac{q_i}{2},\nicefrac{p_i}{2})}{1-\Phi_G(\nicefrac{p_i\lambda^2}{2})}
\label{e417}
\end{eqnarray}
where,  $C \sim \chi^2_{p_i+q_i}$, $\Phi_B(\bullet)$ and $\Phi_G(\bullet)$ are the cumulative distribution functions of a Beta and Gamma random variables respectively. In equation \eqref{e417}, $p_i$ and $q_i$ are the degrees of freedom, and along with $\lambda$  form the parameters for $\kappa_\eta^{(p_i,q_i,\lambda)}$. For a profile of dimension `n', define $n_t= \dfrac{2^J}{2^l_t}$
\begin{eqnarray}
p_i &=& n-n_t \\
q_i &=& n_t
\label{e418}
\end{eqnarray}
where, $l_t$ is the level level up to which a thresholding rule is applied.
\end{proposition}
\begin{proof}
The proof of Proposition \eqref{p01} is shown here. First, it should be noted the density of $\kappa_\eta^{(p_i,q_i,\lambda)}$ is a convolution of the densities for $\left] \chi_{p_i}^2 \right[_\lambda$ and $\chi_{q_i}^2$, that is,
\begin{eqnarray}
f_{\kappa_\eta^{(p_i,q_i,\lambda)}}(x) = f_{\chi^{2}_{q_i}}(x) \otimes f_{\left] \chi^{2}_{p_i} \right[_\lambda}(x)
\label{e419}
\end{eqnarray}
We know,
\begin{eqnarray}
f_{\chi^{2}_{q_i}}(x)  &=& \dfrac{1}{2^{\nicefrac{q_i}{2}}\Gamma(\nicefrac{q_i}{2})} x^{\nicefrac{q_i}{2}-1} e^{-\nicefrac{x}{2}} \\
f_{\left] \chi^2_{p_i} \right[_\lambda}(x) &=& \dfrac{x^{\nicefrac{p_i}{2}-1}e^{-\nicefrac{x}{2}}}{2^{\nicefrac{p_i}{2}}[1-\Phi_{G}(\nicefrac{p_i\lambda^2}{2})]\Gamma(\nicefrac{p_i}{2})}
\label{e420}
\end{eqnarray}
In equation \eqref{e420}, `G' is a Gamma random variable with shape and scale parameters as $\nicefrac{p_i}{2}$ and 1 respectively. The density of a truncated $\chi^2$ distribution is shown in appendix \eqref{A1}. To simplify notations in evaluating the convolution integral, define: $S \equiv \kappa_\eta^{(p_i,q_i,\lambda)}, W \equiv \chi^2_{\nicefrac{q_i}{2}}, \textrm{and}\; Z \equiv \left] \chi^2_{p_i} \right[_\lambda $, and
\begin{eqnarray}
S= W+Z = \kappa_\eta^{(p_i,q_i,\lambda)}
\Rightarrow f_S(s) = \displaystyle \int_{j\lambda^2}^{s} f_W(s-t)f_Z(t)dt
\label{e421}
\end{eqnarray}
Hence,
\begin{eqnarray}
f_S(s) &=& \displaystyle \int_{p_i\lambda^2}^{s} \dfrac{(s-t)^{\nicefrac{q_i}{2}-1}e^{-\nicefrac{(s-t)}{2}}}{2^{\nicefrac{q_i}{2}}\Gamma(\nicefrac{q_i}{2})}\dfrac{s^{\nicefrac{p_i}{2}-1}e^{-\nicefrac{s}{2}}}{2^{\nicefrac{p}{2}}\Gamma(\nicefrac{p_i}{2})(1-\phi_G(\nicefrac{p_i\lambda^2}{2}))}dt \\
&=& \dfrac{1}{\Gamma(\nicefrac{p_i}{2})\Gamma(\nicefrac{q_i}{2})2^{\nicefrac{(p_i+q_i)}{2}}} \displaystyle \int_{p_i\lambda^2}^{s} \dfrac{(s-t)^{\nicefrac{q_i}{2}-1}e^{-\nicefrac{s}{2}}t^{\nicefrac{p_i}{2}-1}}{1-\phi_G(\nicefrac{p_i\lambda^2}{2})}dt\\
&=& \dfrac{1}{\Gamma(\nicefrac{p_i}{2})\Gamma(\nicefrac{q_i}{2})2^{\nicefrac{(p_i+q_i)}{2}}(1-\phi_G(\nicefrac{p_i\lambda^2}{2}))} \displaystyle \int_{p_i\lambda^2}^{s} (s-t)^{\nicefrac{q_i}{2}-1}e^{-\nicefrac{s}{2}}t^{\nicefrac{p_i}{2}-1}dt\\
&=& \dfrac{1}{\Gamma(\nicefrac{p_i}{2})\Gamma(\nicefrac{q_i}{2})2^{\nicefrac{(p_i+q_i)}{2}}(1-\phi_G(\nicefrac{p_i\lambda^2}{2}))} \displaystyle \int_{p_i\lambda^2}^{s} (s-t)^{\nicefrac{q_i}{2}-1}e^{-\nicefrac{s}{2}}t^{\nicefrac{p_i}{2}-1}dt\\
&=& \dfrac{e^{-\nicefrac{s}{2}}s^{\nicefrac{p_i}{2}-1}s^{\nicefrac{q_i}{2}-1}}{\Gamma(\nicefrac{p_i}{2})\Gamma(\nicefrac{q_i}{2})2^{\nicefrac{(p_i+q_i)}{2}}(1-\phi_G(\nicefrac{p_i\lambda^2}{2}))} \displaystyle \int_{j\lambda^2}^{s} (1-\nicefrac{t}{s})^{\nicefrac{q_i}{2}-1}(\nicefrac{t}{s})^{\nicefrac{p_i}{2}-1}dt
\label{e422}
\end{eqnarray}
Using a substitution, $r=\nicefrac{t}{s}$, the above integral becomes,
\begin{eqnarray}
&=& \dfrac{e^{-\nicefrac{s}{2}}s^{\nicefrac{(p_i+q_i)}{2}-1}}{\Gamma(\nicefrac{p_i}{2})\Gamma(\nicefrac{q_i}{2})2^{\nicefrac{(p_i+q_i)}{2}}(1-\Phi_G(\nicefrac{p_i\lambda^2}{2}))} \displaystyle \int_{\nicefrac{p_i\lambda^2}{s}}^{1} (1-r)^{\nicefrac{q_i}{2}-1}(r)^{\nicefrac{p_i}{2}-1}dr\\
\label{a53}
&=& \dfrac{e^{-\nicefrac{s}{2}}s^{\nicefrac{(p_i+q_i)}{2}-1}}{\Gamma(\nicefrac{(p_i+q_i)}{2})2^{\nicefrac{(p_i+q_i)}{2}}}\dfrac{(1-\Phi_B(j\lambda^2,\nicefrac{p_i}{2},\nicefrac{q_i}{2}))}{(1-\Phi_G(\nicefrac{p_i\lambda^2}{2}))}\\
\label{a54}
&=& f_C(s)\dfrac{(1-\Phi_B(\nicefrac{p_i\lambda^2}{s},\nicefrac{p_i}{2},\nicefrac{q_i}{2}))}{(1-\Phi_G(\nicefrac{p_i\lambda^2}{2}))}
\label{a55}
\end{eqnarray}
where $C \sim \chi^2 _{(p_i+q_i)}$.
\begin{flushright}
\textbf{Q.E.D}
\end{flushright}
\subsection{Mean and Variance of $\kappa_\eta^{(p_i,q_i,\lambda)}$}
\label{ap.sec.5}
Consider,
\begin{eqnarray}
X_i \sim Z_i \mathbb{I}_{\{|Z_i| > \lambda\}}
\label{a57}
\end{eqnarray}
where $Z_i \sim N(0,1)$. Under normality assumptions,
\begin{eqnarray}
\sum_{i=1}^{p_i} X_i^2 \sim \left] \chi^2_{p_i} \right[_\lambda
\label{a58}
\end{eqnarray}
Thus,
\begin{eqnarray}
\E[X_i^2] &=& \displaystyle \int_{-\infty}^{\infty} z_i^2 \dfrac{1}{\sqrt{2\pi}}e^{-\nicefrac{z_i}{2}}\mathbb{I}_{\{|Z_i| > \lambda\}}dz_i\\
&=& \displaystyle \int_{-\infty}^{-\lambda}  z_i^2 \dfrac{1}{\sqrt{2\pi}}e^{-\nicefrac{z_i}{2}} dz_i + \int_{\lambda}^{\infty}  z_i^2 \dfrac{1}{\sqrt{2\pi}}e^{-\nicefrac{z_i}{2}} dz_i
\label{a59}
\end{eqnarray}
Define, $Y= \nicefrac{Z_i}{2}$, then,
\begin{eqnarray}
\E[X_i^2] &=& \displaystyle \int_{-\infty}^{-\nicefrac{\lambda}{2}}  4y^2 \dfrac{1}{\sqrt{2\pi}}e^{-y} (2)dy + \int_{\nicefrac{\lambda}{2}}^{\infty}  4y^2 \dfrac{1}{\sqrt{2\pi}}e^{-y} (2)dy\\
&=& \dfrac{4\Gamma(3)}{\sqrt{2\pi}} \displaystyle \int_{\nicefrac{\lambda}{2}}^{\infty} \dfrac{y^{3-1}e^{-y}}{\Gamma(3)}(2dy)\\
&=& \dfrac{8\Gamma(3)}{\sqrt{2\pi}}[1-\Phi_G(\nicefrac{\lambda}{2},3,1)] \triangleq \mu \\
\Rightarrow \E[ \left] \chi^2_{p_i} \right[ ] &=& p_i\mu
\label{a60}
\end{eqnarray}
\begin{flushright}
\textbf{Q.E.D}
\end{flushright}
To find the variance, define, $W= \displaystyle \sum_{i=1}^{p_i} X_i^2$. Then, due to independence of wavelet coefficients,
\begin{eqnarray}
\Var[W]&=& p_i\Var[X_i^2]\\
&=& p_i(\E[X_i^4] - \mu^2) \\
\E[X_i^4] & =& \displaystyle \int_{-\infty}^{\infty} z_i^{4} \mathbb{I}_{\{|z_i| > \lambda\}} e^{-\nicefrac{z_i}{2}}dz_i
\label{a61}
\end{eqnarray}
Define, $Y= \nicefrac{Z_i}{2}$, then,
\begin{eqnarray}
\E[X_i^4] &=& \displaystyle \int_{-\infty}^{-\nicefrac{\lambda}{2}}  16y^4 \dfrac{1}{\sqrt{2\pi}}e^{-y} (2)dy + \int_{\nicefrac{\lambda}{2}}^{\infty}  16y^5 \dfrac{1}{\sqrt{2\pi}}e^{-y} (2)dy\\
&=& \dfrac{16\Gamma(5)}{\sqrt{2\pi}} \displaystyle \int_{\nicefrac{\lambda}{2}}^{\infty} \dfrac{y^{5-1}e^{-y}}{\Gamma(5)}(2dy)\\
&=& \dfrac{32\Gamma(5)}{\sqrt{2\pi}}[1-\Phi_G(\nicefrac{\lambda}{2},5,1)]
\label{a62}
\end{eqnarray}
Thus,
\begin{eqnarray}
\sigma^2 &=& \dfrac{32p_i\Gamma(5)}{\sqrt{2\pi}}[1-\Phi_G(\nicefrac{\lambda}{2},5,1)] -p_i\mu^2
\label{a63}
\end{eqnarray}
From equation \eqref{e415}, the mean of $\kappa_\eta^{(p_i,q_i,\lambda)}$ can be given as,
\begin{eqnarray}
\E[\kappa_\eta^{(p_i,q_i,\lambda)}] &=& \E[\left] \chi^2_{p_i} \right[_\lambda] + E[\chi^2_{q_i}] \\
& = & p_i\mu + q_i  \triangleq \mu_{\kappa_{\eta_i}}
\label{a64}
\end{eqnarray}
Similarly, the variance of $\kappa_\eta^{(p_i,q_i,\lambda)}$  can be given as,
\begin{eqnarray}
\sigma^2_{\kappa_{\eta_i}} &=& \Var [\left] \chi^2_{p_i} \right[_\lambda] + \Var [ \chi^2_{q_i}]-2\Cov(\left] \chi^2_{p_i} \right[_\lambda ,\chi^2_{q_i})\\
\Rightarrow \sigma^2_{\kappa_{\eta_i}} &=& \Var [\left] \chi^2_{p_i} \right[_\lambda] + \Var [ \chi^2_{q_i}]  \quad\quad\quad \textrm{(independent coefficients)} \\
&=& \sigma^2 + 2q_i
\label{a65}
\end{eqnarray}
\end{proof}
\section{Distribution of $\kappa_\eta^{(p,q,\lambda)}$ }
Due to the independence of wavelet coefficients and the assumptions that samples form `T' treatments are independent of each other, the distribution of $\kappa_\eta^{(p,q,\lambda)}$ where,
\begin{eqnarray}
p &=& \sum_{i=1}^{T} p_i \\
q &=& \sum_{i=1}^{T} q_i
\label{a66}
\end{eqnarray}
and $`\lambda$' is the threshold used in wavelet shrinkage for each treatment. The expected value and variance can be given as,
\begin{eqnarray}
\mu_{\kappa_\eta^{p,q}} &=& \sum_{i=1}^{T} \mu_{\kappa_{\eta_i}}\\
&=& \sum_{i=1}^{T} p_i \mu \\
&=& p\mu + q\\
\label{a671}
\sigma^2_{\kappa_\eta^{p,q}} &=& \sum_{i=1}^{T} \sigma^2_{\kappa_{\eta_i}}\\
&=& \dfrac{32p\Gamma(5)}{\sqrt{2\pi}}[1-\Phi_G(\nicefrac{\lambda}{2},5,1)] -p\mu^2 + 2q
\label{a67}
\end{eqnarray}
Thus, the distribution of the statistic in equation \eqref{e18} is $\kappa_\eta^{(p,q,\lambda)}$ with $m_i=p_i+q_i$. The critical values for any test involving this statistic can therefore be obtained using the distribution for $\kappa_\eta^{(p,q,\lambda)}$.
\section{Approximations for $\kappa_\eta^{(p,q,\lambda)}$ }
Using Central Limit Theorem (CLT) \cite{rosenthal2006first}, a normal approximation for the statistic $\kappa_\eta^{(p,q,\lambda)}$ can be obtained as,
\begin{eqnarray}
\dfrac{(S_{n}-\mu_{\kappa_\eta^{p,q}})}{\sigma_{\kappa_\eta^{p,q}}} \sim \xi
\label{e57}
\end{eqnarray}
where $\xi \sim N(0,1)$ and $S_n \sim \kappa_\eta^{(p,q,\lambda)}$.
\subsection{Chi-squared Approximation}
With high-dimensional profiles, as $n \rightarrow \infty \Rightarrow p \rightarrow \infty , q \rightarrow \infty$, using equation \eqref{e420}, it can be shown that,
\begin{eqnarray}
\lim_{p \rightarrow \infty} \left] \chi^2_{p} \right[_\lambda &\xrightarrow[]{d}& \chi^2_{p}\\
\Rightarrow \lim_{p \rightarrow \infty} \kappa_\eta^{(p,q,\lambda)} &\xrightarrow[]{d}& \chi^2_{p+q}
\label{a68}
\end{eqnarray}
An asymptotic moment based $\chi^2$-approximation can be deduced from \eqref{a68}. Since the first moment of a Chi-squared distribution is its parameter, using equation (77) and \eqref{a68}, it can be inferred that.
\begin{eqnarray}
\kappa_\eta^{(p,q,\lambda)} \sim \chi^2_{\lceil p\mu+q \rceil}
\label{a69}
\end{eqnarray}
In equation \eqref{a69}, the degrees of freedom is rounded up to the nearest integer. However, the actual fractional degrees of freedom can be used instead.
\subsection{Binomial-Normal Approximation}
Recall from equation \eqref{e415}, the statistic is comprised of two components. The first corresponds to the coefficients that are thresholded. These coefficients can be modeled using a binomial distribution. Under the assumption of normality and universal hard threshold, the probability of a non-zero wavelet coefficient ($\pi$) after a threshold is,
\begin{eqnarray}
\pi = \Pr(|\theta_i| \geq \lambda) &=& \Pr(\theta_i > \lambda) + \Pr(-\theta_i >\lambda)\\
&=& \Pr(\theta_i > \lambda) + \Pr(\theta_i < - \lambda) \\
&=& 2\Pr(\theta_i > \lambda) \\
&=& 2(1-\Pr(\theta_i \leq \lambda)) \\
&=& 2(1- \Phi(1.18\sqrt{J}))
\label{e418}
\end{eqnarray}
where $\Phi(\bullet)$ is the CDF of a standard normal distribution. Since $p_i$ (in equation \eqref{e415}) is the number of non-zero coefficients after a threshold, it can be approximated using the expectation of a binomial random variable.
\begin{eqnarray}
p_i &=& \E [B_i]\\
\Rightarrow p_i &=& (n-n_t)\pi
\label{a70}
\end{eqnarray}
where, $B_i \sim \textrm{Bin}(n-n_t,\pi)$ with $n_t=\nicefrac{n}{2^{l_t}}$. The Binomial-Normal approximation for $\kappa_\eta^{p,q}$ is,
\begin{eqnarray}
\dfrac{(S_{n}-\mu_{\kappa_\eta^{p,q}})}{\sigma_{\kappa_\eta^{p,q}}} \sim \xi
\label{a71}
\end{eqnarray}
with $p= T(n-n_t)\pi$.\\
\subsection{Binomial-Chi-squared Approximation}
A binomial distribution based Chi-squared approximation for $\kappa_\eta^{(p,q,\lambda)}$ can be obtained by using $p=T(n-n_t)\pi$ in equation \eqref{a69}. The fractional degrees of freedom can be used without rounding with this approximation as well.
\section{Conclusion}
A wavelet based test statistic primarily used for testing differences in high-dimensional profiles was proposed in this report. The focus of this report was to discuss the statistical properties of the proposed statistic. In particular, its exact distribution seemed feasible and it became the premise of this report. The existence of an exact distribution for a test statistic is crucial in finding critical and p-values. Also, such exact distributions offer a strong foundation for future research in this area. Therefore, the discussions in this report involved proofs and derivations pertaining to the exact distribution. In addition to the exact distribution, approximating distributions were also provided. The statistic proposed in this report assumes normality and in the future, it may be possible to obtain similar proofs for an exact distribution when the assumption is violated. The report did not discuss any simulations or visualizations as it was intended to be purely theoretical.
\appendix
\section{Appendix}
\subsection{Density and MGF of a $ \left] \chi^2_{M} \right[_\lambda$ Distribution}
\label{A1}
Consider a truncated normal random variable with mean '0' and standard deviation '1', then its density is given by,
\begin{eqnarray}
f_X(x;0,1,\lambda) &=& \dfrac{\phi(x)}{2\Phi(-\lambda)}\mathbb{I}_{(-\infty,-\lambda]}+\dfrac{\phi(x)}{2(1-\Phi(\lambda))}\mathbb{I}_{[\lambda,\infty)}
\label{a28}
\end{eqnarray}
where $\mathbb{I}(\bullet)$ is the indicator function, $\Phi(\bullet)$ and $\phi(\bullet)$ are the cumulative distribution function and the density function of a standard normal random variable respectively. That is,
\begin{eqnarray}
\phi(x) &=& \dfrac{1}{\sqrt{2\pi}}e^{-\nicefrac{x^2}{2}}\\
\Phi(x) &=& \displaystyle \int_{-\infty}^{x}  \dfrac{1}{\sqrt{2\pi}}e^{-\nicefrac{t^2}{2}}dt
\label{a29}
\end{eqnarray}
Let,
\begin{eqnarray}
C_M &=& \displaystyle \sum_{i=1}^{M} X_{i}^2 \\
\label{a35}
\end{eqnarray}
where, $supp(C_M) \in [M\lambda^2,\infty)$, and $X_i$ is a truncated normal random variable with support in $(-\infty,-\lambda] \cup [\lambda,\infty)$. Then,
\begin{eqnarray}
F_{C_M}(c)&=& \Pr(C_M \leq c)\\
&=& \displaystyle \int_{X_1} \int_{X_2} \ldots \int_{X_M} \prod_{i=1}^{M} f_{X_i}(x_i)dx_i
\label{a36}
\end{eqnarray}
Using \eqref{a28},
\begin{eqnarray}
F_{C_M}(c)&=& \dfrac{1}{\sqrt{[2(2\pi(1-\Phi(\lambda)))]^M}}\displaystyle \int_{X_1} \int_{X_2} \ldots \int_{X_M}  e^{\displaystyle -\sum_{i=1}^{M}\dfrac{x_{i}^2}{2}}\prod_{i=1}^{M} dx_i
\label{a37}
\end{eqnarray}
Geometrically, one can interpret $C_M$ as $\sqrt{R}$, where $R$ is the radius of an $(M-1)$-sphere (the surface of the $M$-ball)\cite{henderson2005experiencing} centered at $ \mathbb{O} \in \mathbb{R}^M$, the origin.\\
\\
Define, $\xi = (2\sqrt{2\pi}(1-\Phi(\lambda)))^{-1}$, then the integral in \eqref{a37} can be expressed in terms of the surface area of a (M-1) sphere, $A_M$,
\begin{eqnarray}
F_{C_M}(c)&=& \dfrac{1}{\xi^M} \displaystyle \int_{\sqrt{M}\lambda}^{\sqrt{c}} Ae^{-\nicefrac{R^2}{2}}dR
\label{a38}
\end{eqnarray}
where,
\begin{eqnarray}
R &\in & [\sqrt{M}\lambda,\sqrt{c}]\\
\label{a39}
A_M &=& \dfrac{MR^{M-1}\pi^{\nicefrac{M}{2}}}{\Gamma\left(\dfrac{M}{2}+1\right)}
\label{a40}
\end{eqnarray}
Using equations \eqref{a39} and \eqref{a40} in \eqref{a38} with the property of gamma functions $\Gamma
\left(\dfrac{M}{2}+1\right) = \dfrac{M}{2}\Gamma\left(\dfrac{M}{2}\right)$,
\begin{eqnarray}
F_{C_M}(c) &=& \dfrac{2\pi^{\nicefrac{M}{2}}}{\xi^{M} \Gamma \left(\dfrac{M}{2}\right)} \displaystyle \int_{\sqrt{M}\lambda}^{\sqrt{c}} R^{M-1}e^{-\nicefrac{R^2}{2}}dR
\label{a41}
\end{eqnarray}
Using the second fundamental theorem of calculus in equation \eqref{a41}, the density of $C_M$ is derived as,
\begin{eqnarray}
f_{C_M}(c)= \dfrac{\pi^{\nicefrac{M}{2}}c^{\nicefrac{M}{2}-1}e^{\nicefrac{c}{2}}}{\xi^M \Gamma\left(\dfrac{M}{2}\right)}
\label{a42}
\end{eqnarray}
Replacing $\xi$, the density can be obtained as,
\begin{eqnarray}
f_{C_M}(c) &\propto & \dfrac{c^{\nicefrac{M}{2}-1}e^{\nicefrac{c}{2}}}{2^{\nicefrac{M}{2} + 1} (1-\Phi(\lambda))^M \Gamma\left(\dfrac{M}{2}\right)} \\
&=& K \dfrac{c^{\nicefrac{M}{2}-1}e^{\nicefrac{c}{2}}}{2^{\nicefrac{M}{2} + 1} (1-\Phi(\lambda))^M \Gamma\left(\dfrac{M}{2}\right)}
\label{a43}
\end{eqnarray}
Equation \eqref{a43} represents an unnormalized density and by integrating $f_{C_M}$ over the support of $C_M$, the actual normalized density, $f_{C_M}$, can be obtained. Therefore, we require,
\begin{eqnarray}
K \displaystyle \int_{M\lambda^2}^{\infty} f_{C_M}(c)dc  =1
\end{eqnarray}
where `K' is the normalizing constant. That is,
\begin{eqnarray}
K \displaystyle \int_{M\lambda^2}^{\infty} \dfrac{c^{\nicefrac{M}{2}-1}e^{\nicefrac{c}{2}}}{2^{\nicefrac{M}{2} + 1} (1-\Phi(\lambda))^M \Gamma\left(\dfrac{M}{2}\right)} dc  =1
\label{a44}
\end{eqnarray}
Let $b=\dfrac{c}{2}$, then
\begin{eqnarray}
K \displaystyle \int_{M\lambda^2}^{\infty} f_{C_M}(c)dc =K \displaystyle  \dfrac{1}{2(1-\Phi(\lambda))^M} \int_{M\nicefrac{\lambda^2}{2}}^{\infty} b^{\nicefrac{M}{2}-1}e^{-b}db \\
\label{a45}
\end{eqnarray}
The integral in equation \eqref{a45} corresponds to the upper incomplete gamma function, $\Gamma\left(\nicefrac{M}{2},\nicefrac{M\lambda^2}{2}\right)$. Equivalently, the upper incomplete gamma function can be expressed using the cumulative distribution function of a Gamma(s,1) random variable, $\Phi_G(\bullet)$. Here, `s' is the scale parameter. Using, $\Gamma\left(s,x\right)=[1-\Phi_G(x)]\Gamma(s)$ in equation \eqref{a45} and setting it equal to 1, the normalizing constant can be found as,
\begin{eqnarray}
K &=& \dfrac{2(1-\Phi(\lambda))^M}{1-\Phi_G(\nicefrac{M\lambda^2}{2})}
\label{a46}
\end{eqnarray}
Using equation \eqref{a46} in \eqref{a43}, the density can be deduced as,
\begin{eqnarray}
f_{C_M}(c) &=& \dfrac{c^{\nicefrac{M}{2}-1}e^{-\nicefrac{c}{2}}}{2^{\nicefrac{M}{2}}[1-\Phi_{G}(\nicefrac{M\lambda^2}{2})]\Gamma(\nicefrac{M}{2})}
\label{a47}
\end{eqnarray}
Due to independence,
\begin{eqnarray}
\E[e^{C_{M}t}]&=&\E\left[e^{\sum_{i=1}^{M}Y_{i}t}\right]\\
&=& \displaystyle \prod_{i=1}^{M} \E[e^{Y_{i}t}]
\label{a48}
\end{eqnarray}
and since $Y_i \sim  \left] \chi^2_{1} \right[_\lambda$, the moment generating function can be found as,
\begin{eqnarray}
M_{C_M}(t) &=& (1-2t)^{-\nicefrac{M}{2}}, \quad t \leq 0
\label{a49}
\end{eqnarray}
\bibliographystyle{natbib}
\bibliography{DWT-TR-Main}
\end{document}